\documentclass{amsart}
\title[Some new sums of $q$-trigonometric...]{Some new sums of $q$-trigonometric and related functions through a theta product of Jacobi}
\usepackage{amssymb,amsmath,amsthm,epsfig,graphics,latexsym}
 \theoremstyle{definition}

  \theoremstyle{plain}
  \newtheorem{lemma}      {Lemma}
  
  \newtheorem{theorem}    {Theorem}
  
  \newtheorem{corollary}  {Corollary}

  \theoremstyle{remark}
  
   \newtheorem{remark}{{\bf Remark}}
  \newcommand{\te}{\theta}
  \newcommand{\fr}{\frac}

  \DeclareMathOperator{\Ct}{Ct}
\begin{document}
  \author[M. El Bachraoui and J. S\'{a}ndor]{Mohamed El Bachraoui and J\'{o}zsef S\'{a}ndor}
  \address{Dept. Math. Sci,
 United Arab Emirates University, PO Box 15551, Al-Ain, UAE}
 \email{melbachraoui@uaeu.ac.ae}
 \address{Babes-Bolyai University, Department of Mathematics and Computer Science, 400084 Cluj-Napoca, Romania}
 \email{jsandor@math.ubbcluj.ro}
 \keywords{$q$-trigonometric functions; $q$-digamma function; transcendence.}
\subjclass{33B15, 11J81, 33E05, 11J86}
 \begin{abstract}
We evaluate some finite and infinite sums involving $q$-trigonometric and $q$-digamma functions. Upon letting
$q$ approach $1$, one obtains corresponding sums for the classical trigonometric and the digamma functions.
Our key argument is a theta product formula of Jacobi
and Gosper's $q$-trigonometric identities.
 \end{abstract}
  \date{\textit{\today}}
  \maketitle
\section{Introduction}\label{sec-introduction}
Throughout we let $\tau$ be a complex number in the upper half plane and let $q=e^{\pi i\tau}$.
Note that the assumption $\mathrm{Im}(\tau)>0$ implies
that $|q|<1$.
The $q$-shifted factorials of a complex number $a$ are defined by
\[
(a;q)_0= 1,\quad (a;q)_n = \prod_{i=0}^{n-1}(1-a q^i),\quad
(a;q)_{\infty} = \lim_{n\to\infty}(a;q)_n.
\]
For convenience we write
\[
(a_1,\ldots,a_k;q)_n = (a_1;q)_n\cdots (a_k;q)_n,\quad
(a_1,\ldots,a_k;q)_{\infty} = (a_1;q)_{\infty} \cdots (a_k;q)_{\infty}.
\]
The $q$-gamma function is given by
\[
\Gamma_q(z) = \dfrac{(q;q)_\infty}{(q^{z};q)_\infty} (1-q)^{1-z} \quad (|q|<1)
\]
and it is well-known that $\Gamma_q (z)$ is a $q$-analogue for the gamma function $\Gamma (z)$, see
%For details on the function $\Gamma_q(z)$ we refer to
\cite{Andrews-Askey-Roy, Askey, Gasper-Rahman, Jackson-1, Jackson-2} for details on the function $\Gamma_q(z)$.
The digamma function $\psi(z)$  and the $q$-digamma function $\psi_q (z)$ are given by
\[
%\begin{split}
\psi(z) = \big(\log\Gamma(z)\big)' = \fr{\Gamma'(z)}{\Gamma(z)} \quad\text{and\quad}
\psi_q (z) = \big(\log\Gamma_q(z)\big)' = \fr{\Gamma_q '(z)}{\Gamma_q (z)}.
\]
By Krattenthaler and Srivastava~\cite{Krattenthaler-Srivastava} one has
$\lim_{q\to 1} \psi_q(z) = \psi(z)$, showing that the function $\psi_q(z)$ is the $q$-analogue for
 the function $\psi (z)$.
Jacobi first theta function is defined as follows:
\[
%\begin{split}
\theta_1(z \mid \tau) = 2\sum_{n=0}^{\infty}(-1)^n q^{(2n+1)^2/4}\sin(2n+1)z
= i q^{\frac{1}{4}}e^{-iz} (q^2 e^{-2iz},e^{2iz},q^2; q^2)_{\infty}.
\]
Jacobi theta functions have been extensively studied by mathematicians during the last two centuries with hundreds of
properties and formulas as a result. 
Standard references on theta functions include Lawden~\cite{Lawden} and Whittaker~and~Watson~\cite{Whittaker-Watson}.
Among the well-known properties of the function $\theta_1(z\mid\tau)$ which we need in this paper 
we have
\begin{equation}\label{theta-cot}
\fr{\theta_1'(z|\tau)}{\theta_1(z|\tau)} = \cot z + 4\sum_{n=1}^{\infty}\fr{q^{2n}}{1-q^{2n}} \sin (2nz).
\end{equation}
Gosper~\cite{Gosper} introduced $q$-analogues of $\sin z$ and $\cos z$ as follows
\begin{equation}\label{sine-cosine-q-gamma}
\begin{split}
\sin_q \pi z
&=
q^{\fr{1}{4}} \Gamma_{q^2}^2\left(\fr{1}{2}\right) \frac{q^{z(z-1)}}{\Gamma_{q^2}(z) \Gamma_{q^2}(1-z)} \\
\cos_q \pi z
&=
\Gamma_{q^2}^2\left(\fr{1}{2}\right) \fr{q^{z^2}}{\Gamma_{q^2}\left(\fr{1}{2}-z \right) \Gamma_{q^2}\left(\fr{1}{2}+z\right)}.
\end{split}
\end{equation}
and proved that %the functions $\sin_q z$ and $\cos_q z$ are related to the function $\theta_1(z\mid \tau')$ as follows:
\begin{equation}\label{sine-cosine-theta}
\begin{split}
\sin_q (z) = \frac{\theta_1(z\mid \tau')}{\theta_1\left( \frac{\pi}{2}\bigm| \tau' \right)} \qquad \text{and \quad}
\cos_q (z) = \frac{\theta_1\left( z+\frac{\pi}{2} \bigm| \tau' \right)}
{\theta_1 \left( \frac{\pi}{2} \bigm| \tau' \right)} \quad \quad (\tau' = \frac{-1}{\tau}).
\end{split}
\end{equation}
It can be shown that
$\lim_{q\to 1}\sin_q z = \sin z$ and $\lim_{q\to 1}\cos_q z = \cos z$.
Moreover, from (\ref{sine-cosine-q-gamma}), one can easily verify by differentiating logarithms that
$\sin_q' (z)$ is the $q$-analogue of $\sin' (z) = \cos z$ and that
$\cos_q' (z)$ is the $q$-analogue of $\cos' (z) = -\sin z$.
We mention that there are known other examples of $q$-analogues for the functions $\sin z$
and $\cos z$, see for instance the book by Gasper~and~Rahman~\cite{Gasper-Rahman}.
A function which is very important for our current purpose is
\begin{equation}\label{Cotan-q}
%For any complex $z$ such that $z\not= k\pi$ for any integer $k$ let the function $\Ct_q(z)$ be defined by
\Ct_q(z) = \fr{\sin_q' z}{\sin_q z}
\end{equation}
for which we clearly have $\lim_{q\to 1} \Ct_q(z) = \cot z$. In addition, by taking in
(\ref{sine-cosine-q-gamma}) logarithms and differentiating with respect to $z$ we get
\begin{equation}\label{reflection}
\psi_{q^2}(z)-\psi_{q^2}(1-z) = (2z-1)\log q - \pi\Ct_q (\pi z),
\end{equation}
which is the $q$-analogue of the well-known reflection formula
\[
\psi(z)-\psi(1-z) = -\pi \cot(\pi z).
\]
Jacobi~\cite{Jacobi} proved that
\begin{equation}\label{MainProd}
\fr{(q^{2n};q^{2n})_{\infty}}{(q^2;q^2)_{\infty}^n}
\prod_{k=-\fr{n-1}{2}}^{\fr{n-1}{2}}\te_1 \left(z+\fr{k\pi}{n} \bigm| \tau \right) =
\te_1(nz \mid n\tau),
\end{equation}
see also Enneper~\cite[p. 249]{Enneper}.
This formula turns out to be equivalent to the following $q$-trigonometric identity
of Gosper~\cite[p. 92]{Gosper}:
\begin{equation}\label{SineProd}
\prod_{k=0}^{n-1}\sin_{q^n}\pi \left(z+\fr{k}{n} \right) =
q^{\frac{(n-1)(n+1)}{12}} \frac{(q;q^2)_{\infty}^2}{(q^n;q^{2n})_{\infty}^{2n}} \sin_q n\pi z
\end{equation}
which he apparently was not aware of as
he stated the identity without proof or reference.
Unlike many of Jacobi's results, the formula (\ref{MainProd})
seems not to have received much attention by mathematicians. This is probably due to the lack of applications.
The authors recently in~\cite{Bachraoui-Sandor} offered a new proof for~(\ref{SineProd}) and as an application they established a
$q$-analogue for the Gauss multiplication formula for the gamma function as well as for 
an identity of S\'{a}ndor~and~T\'{o}th~\cite{Sandor-Toth} for a short product on Euler gamma function.
Our purpose in this note is to apply~(\ref{SineProd}) in order to evaluate finite and infinite sums involving
the function $\Ct_q (z)$ along with
 the functions $h_{q,M,a}(k)$ and $f_{q,M,a}(k)$ both defined on integers $k$ as follows:
\begin{equation}\label{h-f}
\begin{split}
h_{q,M,a}(k) &= \fr{1}{\pi}\Big(
(\log q)\fr{2k+a-2M}{2M}-\psi_q \big(\fr{2k+a}{2M} \big) - \psi_q \big(1-\fr{2k+a}{2M} \big) \Big) \\
f_{q,M,a}(k) &= \sum_{n=1}^{\infty}\fr{q^{\fr{2n}{M}}}{1-q^{\fr{2n}{M}}}\sin \fr{(2k+a)n\pi}{M}.
\end{split}
\end{equation}
More specifically, we shall prove the following main results which are new, up 
to the authors' best knowledge.
\begin{theorem}\label{thm-main-1}
Let $M>1$ be an integer and let $a$ be an odd integer. Then

\noindent
\emph{(a)\ }
\[
\sum_{n=1}^{\infty} \fr{1}{n} \Ct_q\Big(\fr{(2n+a)\pi}{2M}\Big)
= -\fr{1}{M} \sum_{k=1}^M  \Ct_q\Big(\fr{(2k+a)\pi}{2M}\Big) \psi\big(\fr{k}{M}\big).
\]
\noindent
\emph{(b)\ } The function $h_{q,M,a}(k)$ is periodic with period $M$ and we have
\[
\sum_{n=1}^{\infty} \fr{h_{q,M,a}(n)}{n} = -\fr{1}{M} \sum_{k=1}^M h_{q,M,a}(k) \psi\big(\fr{k}{M}\big).
\]
\noindent
\emph{(c)\ } The function $f_{q,M,a}(k)$ is periodic with period $M$ and we have
\[
\sum_{n=1}^{\infty} \fr{f_{q,M,a}(n)}{n} = -\fr{1}{M} \sum_{k=1}^M f_{q,M,a}(k) \psi\big(\fr{k}{M}\big).
\]
\end{theorem}
\begin{theorem}\label{thm-main-2}
Let $M$ be a positive integer and let $a$ be an odd integer. Then
\begin{align*}
\emph{(a)\quad } &
\sum_{k=1}^M \Big( \psi_q \big( \fr{2k+a}{2M} \big) - \psi_q \big( 1- \fr{2k+a}{2M} \big) \Big)
= \fr{a+1}{2} \log q. \\
\emph{(b)\quad } &
\sum_{k=1}^{M}\Big( \psi_{q}\big( \fr{4k+a}{4M} \big) - \psi_{q}\big( 1-\fr{4k+a}{4M} \big) \Big)  \\
& \quad = \fr{(a+2)\log q}{4} - M\pi \Ct_{q^{1/(2M)}}\big(\fr{a\pi}{4}\big) \\
& \quad =   \begin{cases}
\fr{(a+2)\log q}{4} -\fr{\log q}{4}\fr{\Pi_{q^{1/(4M)}}^2}{\Pi_{q^{1/(2M)}}}
& \text{if\ } a\equiv 1,-3 \pmod{8} \\
\fr{(a+2)\log q}{4} +\fr{\log q}{4}\fr{\Pi_{q^{1/(4M)}}^2}{\Pi_{q^{1/(2M)}}}%{\Pi_{q^{\fr{1}{M}}} }
& \text{if\ } a\equiv -1,3 \pmod{8},
\end{cases} \\
\emph{(c)\quad } &
\sum_{k=1}^{M}\Big( \psi_{q}\big( \fr{6k+a}{6M} \big) - \psi_{q}\big( 1-\fr{6k+a}{6M} \big) \Big) \\
& \quad = \fr{(a+3)\log q}{6} - M\pi \Ct_{q^{1/(2M)}}\big(\fr{a\pi}{6}\big) \\
& \quad =  \begin{cases}
\fr{(a+3)\log q}{6} +\fr{\log q}{3}\fr{\Pi_{q^{1/(6M)}}^{3/2}}{\Pi_{q^{1/(2M)}}^{1/2}}
& \text{if\ } a\equiv 1,-5 \pmod{12} \\
\fr{(a+3)\log q}{6} -\fr{\log q}{3}\fr{\Pi_{q^{1/(6M)}}^{3/2}}{\Pi_{q^{1/(2M)}}^{1/2}}
& \text{if\ } a\equiv -1,5 \pmod{12}. \\
\end{cases}
\end{align*}
\end{theorem}
\begin{remark}\label{rmk-main-1}
By letting $q\to 1$ in Theorem~\ref{thm-main-1} and Theorem~\ref{thm-main-2} one gets
related sums for the functions $\cot z$ and $\psi(z)$. For instance, from Theorem~\ref{thm-main-1}(a)
we obtain for $M>1$ and odd integer $a$
\begin{equation}\label{q-to-1}
\sum_{n=1}^{\infty} \fr{1}{n} \cot\Big(\fr{(2n+a)\pi}{2M}\Big)
= -\fr{1}{M} \sum_{k=1}^M  \cot \Big(\fr{(2k+a)\pi}{2M}\Big) \psi\big(\fr{k}{M}\big)
\end{equation}
and from Theorem~\ref{thm-main-2}(c) we deduce
\[
\sum_{k=1}^M \Big( \psi \big( \fr{6k+a}{6M} \big) - \psi \big( 1- \fr{6k+a}{6M} \big) \Big)
 = - M\pi \cot\big(\fr{a\pi}{6}\big) 
\]
\[
 = \qquad \begin{cases}
- \sqrt{3} M\pi & \text{if\ } a\equiv 1,-5 \pmod{12} \\
 \sqrt{3} M\pi  & \text{if\ } a\equiv -1,5 \pmod{12} \\
 0 & \text{if\ } a\equiv -3,3 \pmod{12}.
 \end{cases}
 \]
\end{remark}
\begin{remark}\label{rmk-transcendence}
By the well-known fact that $\cot r\pi$ is an algebraic number for any rational number $r$ and the relation
(\ref{q-to-1}) we deduce by a result of Adhikari~\emph{et al.}~\cite{Adhikari-et-al} that the sum
$\sum_{n=1}^{\infty} \fr{1}{n} \cot\Big(\fr{(2n+a)\pi}{2M}\Big)$ is either zero or transcendental.
A similar statement can be made about the $q$-analogue of the sum given in 
Theorem~\ref{thm-main-1}(b).
\end{remark}
\noindent
Blagouchine~\cite{Blagouchine} recently evaluated a variety of finite sums involving the digamma function
and the trigonometric functions. For instance, he proved that for any positive integer $M$
\[
\sum_{k=1}^{M-1} \big(\cot\fr{k\pi}{M}\big) \psi\big(\fr{k}{M}\big) = -\fr{\pi(M-1)(M-2)}{6}.
\]
We have the following related contribution.
\begin{theorem}\label{thm-main-3}
For any integer $M>1$ and any odd integer $a$ we have
\[
\sum_{k=1}^{M-1} \Big( \cot\fr{(2k+a)\pi}{2M} + \cot\fr{(2k-a)\pi}{2M} \Big) \psi\big(\fr{k}{M}\big)
= - \sum_{k=1}^{M-1}\big(\cot\fr{k\pi}{M}\big) \cot\fr{(2k+a)\pi}{2M}.
\]
\end{theorem}
\noindent
The rest of the paper is organized as follows. In Section~\ref{Sec:q-trig} we review Gosper's $q$-trigonometry and collect the facts 
which are needed for our discussion. In Section~\ref{sec:proof-main-1} we give the proof of Theorem~\ref{thm-main-1},
Section~\ref{sec:proof-main-2} is devoted to the proof for Theorem~\ref{thm-main-2}, and Section~\ref{sec:proof-main-3}
is devoted to the proof of Theorem~\ref{thm-main-3}.
%\begin{theorem}\label{thm-main-3}
%Let $M>1$ be an integer, let $a$ be an odd integer, and let $\alpha\in\mathbb{Q}\setminus\mathbb{Z}$.
%Then
%\[
%\begin{split}
%\emph{(a)\ } \sum_{n\in\mathbb{Z}\setminus\{0\}} \fr{\cot\pi\big(\fr{2n+a}{2M}\big)}{n^l}
%&=
%\Big(\fr{\pi}{M}\Big)^l \Big( \sum_{k=1}^{M-1}
%\big(\cot\pi\fr{2k+a}{2M}\big) A_{\fr{k}{M},l}+ 2f(M) Z(l) \Big) \\
%\emph{(b)\ } \sum_{n\in\mathbb{Z}\setminus\{0\}} \fr{\cot\pi\big(\fr{2n+a}{2M}\big)}{(n+\alpha)^l}
%&=
%\Big(\fr{\pi}{M}\Big)^l  \sum_{k=1}^{M}
%\big(\cot\pi \fr{2k+a}{2M}\big) A_{\fr{k+\alpha}{M},l}.
%\end{split}
%\]
%\end{theorem}
%
\section{Facts on Gosper's $q$-trigonometry}\label{Sec:q-trig}
\noindent
Just as for the function $\sin z$ and $\cos z$, it is easy to verify that
%Throughout we shall need the following basic facts:
\begin{align}\label{sine-cos-basics}
 \sin_q (\fr{\pi}{2}-z)=\cos_q z,\  \sin_q \pi = 0,\
\sin_q \fr{\pi}{2} = \cos_q 0= 1,  \\
 \sin_q(z+\pi)= -\sin_q z=\sin_q(-z), \ \text{and\ }
-\cos_q(z+\pi) = \cos_q z = \cos_q(-z), \nonumber
\end{align}
from which it follows that for any odd integer $a$,
\begin{equation}\label{sin-cos-aux}
\begin{split}
\sin_{q}\fr{a\pi}{4} &=
\begin{cases}
\sin_q\fr{\pi}{4} & \text{if\ } a\equiv 1, 3 \pmod{8} \\
- \sin_q \fr{\pi}{4}& \text{if\ } a\equiv -1, -3 \pmod{8},
\end{cases} \\
\sin_{q}\fr{a\pi}{6} &=
\begin{cases}
\sin_q\fr{\pi}{6} & \text{if\ } a\equiv 1, 5 \pmod{12}  \\
-\sin_q\fr{\pi}{6} & \text{if\ } a\equiv -1, -5 \pmod{12}  \\ \\
1 & \text{if\ } a\equiv  3 \pmod{12} \\
-1 & \text{if\ } a\equiv  -3 \pmod{12}.
\end{cases}
\end{split}
\end{equation}
Also, by using (\ref{sine-cos-basics}) we have
\begin{align}\label{special-deriv}
\sin_q'\big(\fr{\pi}{2}-z \big)= -\cos_q' z,\ \cos_q'\big(z-\fr{\pi}{2}\big)= \sin_q' z,\\
- \sin_q'(\pi-z) = \sin_q' z,\ \text{and\ }
-\cos_q'(\pi-z) = \cos_q' z  \nonumber
\end{align}
where the derivatives here and in what follows are with respect to $z$.
We can easily see from (\ref{special-deriv}) that for any odd integer $a$ we have
\begin{equation}\label{q-sine-derive-2}
\sin_q' \fr{a\pi}{2} = \cos_q' 0 = 0.
\end{equation}
The following $q$-constant appears frequently in Gosper's manuscript~\cite{Gosper}
\[
\Pi_q = q^{\fr{1}{4}} \fr{(q^2;q^2)_{\infty}^2}{(q;q^2)_{\infty}^2}.
\]
Gosper stated many identities involving $\sin_q z$ and $\cos_q z$ which easily follow just
from the definition and basic properties of other related functions. To mention an example,
he derived that
\begin{equation}\label{q-sine-derive-1}
\sin_q' 0 =- \cos_q'\fr{\pi}{2} =  \fr{-2 \log q}{\pi}\Pi_q.
\end{equation}
On the other hand, Gosper~\cite{Gosper} using the computer facility \emph{MACSYMA} stated without proof a variety of identities involving $\sin_q z$ and $\cos_q z$ and he
asked the natural question whether his formulas hold true.
For instance, based on his conjectures, he stated
\[
\label{q-Double-2} \tag{$q$-Double$_2$}
\sin_q(2z) = \fr{\Pi_q}{\Pi_{q^2}} \sin_{q^2} z \cos_{q^2} z,
\]
\[
\label{q-Double-3} \tag{$q$-Double$_3$}
\cos_q(2z) = (\cos_{q^2} z)^2 - (\sin_{q^2} z)^2,
\]
\[ \label{q-Triple-2} \tag{$q$-Triple$_2$}
\sin_q(3z) = \frac{\Pi_q}{\Pi_{q^3}} (\cos_{q^3} z)^2 \sin_{q^3}z -  (\sin_{q^3}z)^3,
\]
and 
\[   \label{q-Double-5} \tag{$q$-Double$_5$}
\cos_q(2z) = (\cos_{q}z)^4- (\sin_{q}z)^4.
\]
\noindent
A proof for (\ref{q-Double-2}) can be found in Mez\H{o}~\cite{Mezo-1} and proofs for (\ref{q-Double-2})
(\ref{q-Triple-2}), and (\ref{q-Double-5}) were obtained in~\cite{Bachraoui-1, Bachraoui-2, Bachraoui-3}.
Proofs for other identities of Gosper can be found in~\cite{Touk-Houchan-Bachraoui, He-Zhai, He-Zhang}.
Furthermore, Gosper deduced the following special values:
\begin{equation}\label{sin-cos-values}
\begin{split}
\sin_{q^2} \fr{\pi}{4} &= \cos_{q^2} \fr{\pi}{4} = \fr{\Pi_{q^2}^{\fr{1}{2}}}{\Pi_{q}^{\fr{1}{2}}} \\
\left(\sin_{q^3}\fr{\pi}{3} \right)^3 &= \left(\cos_{q^3}\fr{\pi}{6} \right)^3 = \fr{ \left(\fr{\Pi_q}{\Pi_{q^3}} \right)^{\fr{3}{2}}}{\left(\fr{\Pi_q}{\Pi_{q^3}} \right)^2 -1} \\
\left(\sin_{q^3}\fr{\pi}{6} \right)^3 &= \left(\cos_{q^3}\fr{\pi}{3} \right)^3 =
\fr{ 1}{\left(\fr{\Pi_q}{\Pi_{q^3}} \right)^2 -1}.
\end{split}
\end{equation}
\noindent
As to special values for derivatives we have the following list.
%For our next result we need a lemma.
\begin{lemma} \label{lem-1-special}
Let $a$ be an odd integer. Then we have
\[
\begin{split}
\emph{(a)\ } & \sin_{q^2}'\fr{a\pi}{4} =
\begin{cases}
\fr{\log q}{\pi} \fr{\Pi_{q}^{\fr{3}{2}}}{\Pi_{q^2}^{\fr{1}{2}}} & \text{if\ } a\equiv -1, 1 \pmod{8} \\
- \fr{\log q}{\pi} \fr{\Pi_{q}^{\fr{3}{2}}}{\Pi_{q^2}^{\fr{1}{2}}} & \text{if\ } a\equiv -3, 3 \pmod{8}.
\end{cases} \\
%\begin{cases}
% \fr{1}{2} q^{\fr{1}{16}}
%\fr{\log q}{\pi} \fr{(q;q)_{\infty}^2 (q;q^2)_{\infty}^2}{(q^{\fr{1}{2}};q)_{\infty}^3} & \text{if\ } a\equiv 1 \pmod{4} \\
%- \fr{1}{2} q^{\fr{1}{16}}
%\fr{\log q}{\pi} \fr{(q;q)_{\infty}^2 (q;q^2)_{\infty}^2}{(q^{\fr{1}{2}};q)_{\infty}^3} & \text{if\ } a\equiv -1 \pmod{4}.
%\end{cases} \\
\emph{(b)\ } &
\sin_{q^3}'\fr{a\pi}{3} =
\begin{cases}
\fr{\log q}{\pi}
\fr{\Pi_{q^3}^{\fr{1}{3}} (\Pi_q^2 - \Pi_{q^3}^2)^{\fr{2}{3}} (3\Pi_{q^3}^2-\Pi_q^2)}
{\Pi_{q}^2 -  \Pi_{q^3}^2} & \text{if\ } a\equiv -1, 1 \pmod{6}  \\
-\fr{2\log q}{\pi} \Pi_q
 & \text{if\ } a\equiv 3 \pmod{6}
\end{cases} \\
\emph{(c)\ } &
\sin_{q^3}'\fr{a\pi}{6} =
\begin{cases}
-\fr{2 \log q}{\pi} \fr{\Pi_{q}^{\fr{3}{2}} \Pi_{q^3}^{\fr{1}{6}} }{ (\Pi_q^2- \Pi_{q^3}^2)^{\fr{1}{3}}}
& \text{if\ } a\equiv 1, -5 \pmod{12}  \\
\fr{2 \log q}{\pi} \fr{\Pi_{q}^{\fr{3}{2}} \Pi_{q^3}^{\fr{1}{6}} }{ (\Pi_q^2- \Pi_{q^3}^2)^{\fr{1}{3}}}
& \text{if\ } a\equiv -1, 5 \pmod{12} \\
0 & \text{if\ } a\equiv -3, 3 \pmod{12}.
\end{cases}
\end{split}
\]
\end{lemma}
\begin{proof}
(a)\  From (\ref{special-deriv}) we have
\begin{equation}\label{q-sine-derive-quarter}
\sin_{q}' \fr{\pi}{4} = - \cos_{q}' \fr{\pi}{4}.
\end{equation}
On the other hand, from (\ref{q-Double-3})
we have
\[
2\cos_q' 2z = 2(\cos_{q^2} z) \cos_{q^2}' z - 2(\sin_{q^2} z) \sin_{q^2}'z,
\]
where if we let $z=\fr{\pi}{4}$ and use (\ref{q-sine-derive-quarter}) we deduce
\[
2\cos_q'\fr{\pi}{2} = -4 \sin_{q^2}\fr{\pi}{4} \sin_{q^2}'\fr{\pi}{4}.
\]
Now, combine the previous identity with (\ref{q-sine-derive-1}), (\ref{sin-cos-values}), and
(\ref{q-sine-derive-quarter}) to obtain the desired identity for $\sin_{q^2}'\fr{\pi}{4}$. Finally, note that
by (\ref{sine-cos-basics}) we have
\[
\sin_q'\fr{a\pi}{4} = \begin{cases}
\sin_q'\fr{\pi}{4} & \text{if\ } a\equiv \pm 1 \pmod{8}, \\
-\sin_q'\fr{\pi}{4} & \text{if\ } a\equiv \pm 3 \pmod{8}
\end{cases}
\]
to complete the proof of part (a).
\noindent
As to part (b),
from (\ref{special-deriv}), we easily find
\begin{equation}\label{sixth-third}
\cos_q'\fr{\pi}{6} = -\sin_q'\fr{\pi}{3} \ \text{and\ }
\cos_q'\fr{\pi}{3} =  -\sin_q'\fr{\pi}{6}.
\end{equation}
Now differentiating (\ref{q-Double-5}) we have
\[
2\cos_q' 2z = 4 (\cos_q z)^3 \cos_q'z - 4 (\sin_q z)^3 \sin_q'z.
\]
Then taking $z=\fr{\pi}{6}$ in the previous identity, using (\ref{sixth-third}), and simplifying yield
\[
\big(2(\sin_q\fr{\pi}{6})^3 - 1 \big) \sin_q'\fr{\pi}{6} =
2 (\cos_q\fr{\pi}{6})^3 \cos_q'\fr{\pi}{6},
\]
in other words,
\begin{equation}\label{help1-lem-2-special}
\cos_q'\fr{\pi}{6} = \fr{2(\sin_q\fr{\pi}{6})^3 - 1}{2 (\cos_q\fr{\pi}{6})^3} \sin_q'\fr{\pi}{6}.
\end{equation}
On the other hand, differentiate (\ref{q-Triple-2}) to derive
\[
3 \sin_q 3z = \fr{\Pi_q}{\Pi_{q^3}}\big(\sin_{q^3}'z (\cos_{q^3} z)^2 +
2\sin_{q^3}z\cos_{q^3}z\cos_{q^3}'z \big) - 3 (\sin_{q^3} z)^2 \sin_{q^3}'z,
\]
which for $z=\fr{\pi}{3}$ and after simplification gives
\[
-\sin_q'0 = \Big(\fr{\Pi_q}{\Pi_{q^3}} (\cos_{q^3}\fr{\pi}{3})^2 - 3(\sin_{q^3}\fr{\pi}{3})^2 \Big)
\sin_{q^3}'\fr{\pi}{3} + 2 \fr{\Pi_q}{\Pi_{q^3}} \sin_{q^3}\fr{\pi}{3}\cos_{q^3}\fr{\pi}{3}\cos_{q^3}'\fr{\pi}{3}.
\]
It follows by virtue of (\ref{help1-lem-2-special}) and with the help of (\ref{sin-cos-values}) that
\[
\fr{6 \log q}{\pi}\Pi_q =\Big(\fr{\Pi_q}{\Pi_{q^3}} (\cos_{q^3}\fr{\pi}{3})^2 -3(\sin_{q^3}\fr{\pi}{3})^2
+ \fr{4 \fr{\Pi_q}{\Pi_{q^3}} \big(\sin_{q^3}\fr{\pi}{3}\big)^4 \cos_{q^3}\fr{\pi}{3}}
{2 \big(\sin_{q^3}\fr{\pi}{3}\big)^3 - 1} \Big) \sin_{q^3}'\fr{\pi}{3}.
\]
Now solving in the previous identity for $\sin_{q^3}'\fr{\pi}{3}$ and using (\ref{sin-cos-values}),
after a long but straightforward calculation, we derive the desired formula for $\sin_{q^3}'\fr{\pi}{3}$.
Finally, note from (\ref{sine-cos-basics}) that
\[
\sin_q'\fr{a\pi}{3} = \begin{cases}
\sin_q'\fr{\pi}{3} & \text{if\ } a\equiv \pm 1 \pmod{6}, \\
-\sin_q' 0 & \text{if\ } a\equiv \pm 3 \pmod{6}
\end{cases}
\]
to complete the proof of part (b).
The proof for part (c) is similar to the previous parts and it is therefore omitted.
\end{proof}
\noindent
By a combination of Lemma~\ref{lem-1-special} with (\ref{sin-cos-aux}) and
(\ref{sin-cos-values}), we arrive at the main result of this section.
\begin{corollary}\label{cor-special-C}
Let $a$ be an odd integer. Then we have
\[
\begin{split}
\emph{(a)\quad } & \Ct_{q^2} \big(\fr{a\pi}{4}\big) =
\begin{cases}
\fr{\log q}{\pi} \fr{\Pi_{q}^2}{\Pi_{q^2}} & \text{if\ } a\equiv 1, -3 \pmod{8} \\
-\fr{\log q}{\pi} \fr{\Pi_{q}^2}{\Pi_{q^2}} & \text{if\ } a\equiv -1, 3 \pmod{8}.
\end{cases} \\
\emph{(b)\quad } & \Ct_{q^3} \big(\fr{a\pi}{6}\big) =
\begin{cases}
-\fr{2 \log q}{\pi} \fr{\Pi_{q}^{\fr{3}{2}}}{\Pi_{q^3}^{\fr{1}{2}}} & \text{if\ } a\equiv 1, -5 \pmod{12} \\
\fr{2 \log q}{\pi} \fr{\Pi_{q}^{\fr{3}{2}}}{\Pi_{q^3}^{\fr{1}{2}}} & \text{if\ } a\equiv -1, 5 \pmod{12} \\
0 & \text{if\ } a\equiv -3, 3 \pmod{12}.
\end{cases}
\end{split}
\]
\end{corollary}
\section{Proof of Theorem~\ref{thm-main-1}}\label{sec:proof-main-1}
\noindent
We need the following result of Ram~Murty~and~Saradha~\cite{Murty-Saradha} which we record as a lemma.
\begin{lemma}\label{lem-MurtSara}
Let $f$ be any function defined on the integers and with period $M>1$.
\\
\noindent
The infinite series $\sum_{n=1}^{\infty} \fr{f(n)}{n}$  converges if and only if
$\sum_{k=1}^M f(k) = 0$. In case of convergence, we have
\[
\sum_{n=1}^{\infty} \fr{f(n)}{n} = -\fr{1}{M} \sum_{k=1}^M f(k) \psi\big(\fr{k}{M}\big).
%=-\sum_{k=1}^{M-1}\hat{f}(k)\log(1-e^{\fr{2\pi i k}{M}})
\]
%and it is either zero or transcendental.
\end{lemma}
\noindent
\emph{Proof of Theorem~\ref{thm-main-1}(a)}\ Let
\[
f(k) = \Ct_q \Big( \fr{(2k+a)\pi}{2M} \Big)
\]
which is clearly well-defined on the integers and it is periodic with period $M$.
Then based on Lemma~\ref{lem-MurtSara}, all we need is prove that $\sum_{k=1}^{M} f(k) = 0$.
To do so, note that from (\ref{SineProd}) and the fact that $\sin_q(z+\pi) = -\sin_q z$ we find
\begin{equation}\label{SineProd-2}
\prod_{k=1}^{M}\sin_{q^M}\pi \left(z+\fr{k}{M} \right) =
- q^{\frac{(M-1)(M+1)}{12}} \frac{(q;q^2)_{\infty}^2}{(q^M;q^{2M})_{\infty}^{2M}} \sin_q M\pi z.
\end{equation}
Take logarithms and differentiate with respect to $z$ to derive
\[
\pi \sum_{k=1}^M \Ct_{q^M}\Big(z+\fr{k}{M}\Big) = M\pi\Ct_q(M\pi z)= M\pi \fr{\sin_q' M\pi z}{\sin_q M\pi z}.
\]
Now replace $q^M$ with $q$, let $z=\fr{a}{2M}$, and use (\ref{q-sine-derive-2}) to deduce that
that
\begin{equation}\label{sum-Ctq-zero}
\pi \sum_{k=1}^M \Ct_q \Big( \fr{(2k+a)\pi}{2M} \Big) = 0,
\end{equation}
or equivalently,
\[
\sum_{k=1}^{M} f(k) = 0,
\]
as desired.

\noindent
\emph{Proof of Theorem~\ref{thm-main-1}(b)}\ By the relation (\ref{sine-cosine-q-gamma}) we have
\[
\sin_q\pi\big(z+\fr{k}{M}\big) = q^{\fr{1}{4}}\Gamma_{q^2}^2\left(\fr{1}{2}\right)
\frac{q^{(z+\fr{k}{M})(z+\fr{k}{M} -1)}}{\Gamma_{q^2}\big(z+\fr{k}{M} \big)
\Gamma_{q^2}\big(1-z-\fr{k}{M}\big)}
\]
which after taking logarithms and differentiating with respect to $z$ gives
\[
\pi \Ct_q\pi\big(z+\fr{k}{M}\big)
= (\log q)
\big( 2\big(z+\fr{k}{M}\big)-1 \big) - \psi_{q^2}\big(z+\fr{k}{M}\big) - \psi_{q^2}\big(1-z-\fr{k}{M}\big).
\]
Replacing in the previous relation $q^2$ by $q$ and letting $z=\fr{a}{2M}$, we get
\[
\Ct_{q^{1/2}}\big(\fr{(2k+a)\pi}{2M}\big)
=
\fr{1}{\pi}\Big( (\log q)\fr{2k+a-2M}{2M} - \psi_{q}\big(\fr{2k+a}{2M}\big) - \psi_{q}\big(1-\fr{2k+a}{2M}\big) \Big).
\]
As the left-hand side of the previous identity is evidently periodic with period $M$, the same holds for its right-hand side 
which is nothing else but $h_{q,M,a}(k)$. We now claim that $\sum_{k=1}^M h_{q,M,a}(k) = 0$. Indeed,
apply~(\ref{sine-cosine-q-gamma}) to the factors in  identity~(\ref{SineProd-2}), then take logarithms and finally differentiate with respect to $z$ to
obtain
\begin{align}\label{q-psi-key}
M(\log q) & \Big(\sum_{k=1}^M  2\big(z+ \fr{k}{M}\big) -1 \Big) -
\sum_{k=1}^M \Big(\psi_{q^{2M}}\big(z+\fr{k}{M}\big) - \psi_{q^{2M}}\big(1-z-\fr{k}{M}\big) \Big)
\nonumber \\
& = M\pi \Ct_q(M\pi z) .
\end{align}
Next replace $q^{2M}$ by $q$, let $z=\fr{a}{2M}$, and use
(\ref{q-sine-derive-2}) to deduce that
\begin{equation}\label{help1-cor-psiq-1}
\sum_{k=1}^M \Big( (\log q)\fr{2k+a-M}{2M} - \psi_{q}\big(\fr{2k+a}{2M}\big) -
\psi_{q}\big(1-\fr{2k+a}{2M}\big) \Big) = 0.
\end{equation}
That is,
\[
\sum_{k=1}^M h_{q,M,a}(k) = 0,
\]
and the claim is confirmed.
Finally, apply Lemma~\ref{lem-MurtSara} to the function $h_q(M,a,k)$ to complete the proof of part (b).

\noindent
\emph{Proof of Theorem~\ref{thm-main-1}(c)}\ Note that the function $f_{q,M,a}(k)$ is clearly periodic with period $M$.
By virtue of~(\ref{sine-cosine-theta}) and~(\ref{SineProd-2}) and after taking logarithm and differentiating we find
\[
\pi \sum_{k=1}^M \fr{\te_1'\big(\pi z +\fr{k\pi}{M} | \fr{\tau'}{M}\big)}{\te_1\big(\pi z +\fr{k\pi}{M} | \fr{\tau'}{M}\big)}
= M\pi \Ct_q(M\pi z),
\]
which upon substituting $z$ by $\fr{a}{2M}$ and $\tau'$ by $\tau$ yields
\[
\sum_{k=1}^M \fr{\te_1'\big(\fr{(2k+a)\pi}{2M} | \fr{\tau}{M}\big)}{\te_1\big(\fr{(2k+a)\pi}{2M} | \fr{\tau}{M}\big)} = 0.
\]
Now combine the foregoing identity with (\ref{theta-cot}) and the $q$-analogue of (\ref{sum-Ctq-zero}) to derive
\[
\sum_{k=1}^M f_{q,M,a}(k) 
= \sum_{k=1}^M \sum_{n=1}^{\infty}\fr{q^{\fr{2n}{M}}}{1-q^{\fr{2n}{M}}}\sin \fr{(2k+a)n\pi}{M} =0.
\]
Finally apply Lemma~\ref{lem-MurtSara} to the function $f_{q,M,a}(k)$ to complete the proof.
\section{Proof of Theorem~\ref{thm-main-2}}\label{sec:proof-main-2}
\noindent
(a)\ If $M=1$, then the desired formula 
\[
\psi_q\big(1+\fr{a}{2}\big) - \psi_q\big(-\fr{a}{2}\big) = \fr{(a+1)\log q}{2}
\]
follows by virtue of (\ref{reflection}). If $M>1$, then an immediate consequence of (\ref{SineProd-2}) is
\[
\sum_{k=1}^M \Big(\psi_{q}\big(\fr{2k+a}{2M}\big) - \psi_{q}\big(1-\fr{2k+a}{2M}\big)\Big) =
\fr{(a+1) \log q}{2},
\]
which is the desired relation.

\noindent
(b)\ If $M=1$, the statement follows by (\ref{reflection}). Now suppose that $M>1$.
Then by letting in (\ref{q-psi-key}) $z=\fr{a}{4M}$ and after simplification we find
\[
\fr{(a+2)M\log q}{2} -
\sum_{k=1}^M \Big( \psi_{q^{2M}}\big(\fr{4k+a}{4M}\big) - \psi_{q^{2M}}\big(\fr{4(M-k)-a}{4M}\big)\Big)
= M\pi \Ct_q\big(\fr{a\pi}{4}\big).
\]
Then uopn replacing $q^{2M}$ by $q$ and rearranging becomes
\[
\sum_{k=1}^M \Big(\psi_{q}\big(\fr{4k+a}{4M}\big) - \psi_{q}\big(\fr{4(M-k)-a}{4M}\big)\Big)
= \fr{(a+2)\log q}{4} - M\pi \Ct_{q^{1/(2M)}}\big(\fr{a\pi}{4}\big),
\]
which is the first identity of this part. As to the second formula,
simply use Corollary~\ref{cor-special-C} to evaluate the right-hand-side of the previous identity
and rearrange in the appropriate way.

\noindent
(c)\ If $M=1$, the statement follows by (\ref{reflection}). Now suppose that $M>1$.
Let in (\ref{q-psi-key}) $z=\fr{a}{6M}$ and simplify to obtain
\[
\fr{(a+3)M\log q}{3} -
\sum_{k=1}^M \Big(\psi_{q^{2M}}\big(\fr{6k+a}{6M}\big) - \psi_{q^{2M}}\big(\fr{6(M-k)-a}{6M}\big)\Big)
= M\pi \Ct_q\big(\fr{a\pi}{6}\big).
\]
Then replace in the foregoing formula $q^{2M}$ by $q$ and rearrange to get
\[
\sum_{k=1}^M \Big(\psi_{q}\big(\fr{6k+a}{6M}\big) - \psi_{q}\big(\fr{6(M-k)-a}{6M}\big)\Big)
= \fr{(a+3)\log q}{6} - M\pi \Ct_{q^{1/(2M)}}\big(\fr{a\pi}{6}\big).
\]
This establishes the first formula of this part.
Finally apply Corollary~\ref{cor-special-C} to complete the proof.
\section{Proof of Theorem~\ref{thm-main-3}}\label{sec:proof-main-3}
\noindent
In our proof we shall make an appeal to a result of
Weatherby~in~\cite{Weatherby} for which we need the following notation.
For any real number $\alpha$ and any positive integer $l$, let
\[
A_{\alpha,l} := \fr{(-1)^{l-1} \big(\pi\cot(\pi \alpha)\big)^{(l-1)}}{\pi^l(l-1)!}
\]
and let
\[
Z(l) = \begin{cases}
0 & \text{if $l$\ is odd,} \\
\fr{\zeta(l)}{\pi^l} & \text{otherwise.}
\end{cases}
\]
Notice that $Z(l)\in\mathbb{Q}$ for all positive integer $l$.
\begin{lemma}\label{lem-Weatherby}
Let $f$ be an algebraic valued function defined on the integers with period $M>1$ and let $l$ be a positive integer. Then
\[
\sum_{n\in\mathbb{Z}\setminus\{0\}} \fr{f(n)}{n^l} =
\Big(\fr{\pi}{M}\Big)^l \Big( \sum_{k=1}^{M-1} f(k) A_{\fr{k}{M},l}+ 2f(M) Z(l) \Big).
\]
\end{lemma}
\noindent
\emph{Proof of Theorem~\ref{thm-main-3}.\ }
Let
\[
f(k) = \cot \fr{(2k+a)\pi}{2M}.
\]
Clearly $f(k)$ is well-defined on the integers since is $a$ is odd and it is periodic with period $M$.
It is a well-known fact that for any rational number $r$ we have that $\cot\pi r$ is an algebraic number.
Then by virtue of Lemma~\ref{lem-Weatherby} we get
\begin{equation}\label{key-sum-main-3}
\sum_{n\in\mathbb{Z}\setminus\{0\}} \fr{\cot\pi\big(\fr{2n+a}{2M}\big)}{n^l}
=\Big(\fr{\pi}{M}\Big)^l \Big( \sum_{k=1}^{M-1} \big(\cot\pi\fr{2k+a}{2M}\big) A_{\fr{k}{M},l}+ 2f(M) Z(l) \Big),
\end{equation}
which for $l=1$ reduces to
\begin{equation}\label{help2-cor-cot-1}
\begin{split}
\sum_{n\in\mathbb{Z}\setminus\{0\}} \fr{\cot \fr{(2n+a)\pi}{2M}}{n}
&=
\fr{\pi}{M}\sum_{k=1}^{M-1}A_{\fr{k}{M},1}\cot\fr{(2k+a)\pi}{2M} \\
&=
\fr{\pi}{M}\sum_{k=1}^{M-1} \fr{1}{\pi} \big(\cot\fr{k\pi}{M} \big) \cot\fr{(2k+a)\pi}{2M}.
\end{split}
\end{equation}
On the other hand, with the help of the $q$-analogue of Theorem~\ref{thm-main-1}(a) we deduce
\[
%\begin{split}
\sum_{n=1}^{\infty}\fr{\cot \fr{(-2n+a)\pi}{2M}}{-n}
= \sum_{n=1}^{\infty}\fr{\cot \fr{(2n-a)\pi}{2M}}{n}
= -\fr{1}{M} \sum_{k=1}^{M} \Big(\cot\fr{(2k-a)\pi}{2M}\Big)\psi\big(\fr{k}{M}\big),
%\end{split}
\]
which implies that
\[
\begin{split}
\sum_{n\in\mathbb{Z}\setminus\{0\}} \fr{\cot \fr{(2n+a)\pi}{2M}}{n}
&=
\sum_{n=1}^{\infty}\fr{\cot \fr{(2n+a)\pi}{2M}}{n} + \sum_{n=1}^{\infty}\fr{\cot \fr{(2n-a)\pi}{2M}}{n} \\
&=
-\fr{1}{M} \sum_{k=1}^{M} \psi\big(\fr{k}{M}\big)\Big(\cot\fr{(2k+a)\pi}{2M} + \cot\fr{(2k-a)\pi}{2M} \Big).
\end{split}
\]
As
\[\cot\fr{(2M+a)\pi}{2M} + \cot\fr{(2M-a)\pi}{2M}  = \fr{\sin 2\pi}
{\fr{1}{2}\big(\cos\fr{a\pi}{M} - \cos\fr{2M\pi}{M} \big)} = 0,
\]
the $M$-th term in the last summation vanishes and so we get
\begin{equation}\label{help1-cor-cot-1}
\sum_{n\in\mathbb{Z}\setminus\{0\}} \fr{\cot \fr{(2n+a)\pi}{2M}}{n} =
-\fr{1}{M} \sum_{k=1}^{M-1} \psi\big(\fr{k}{M}\big)\Big(\cot\fr{(2k+a)\pi}{2M} + \cot\fr{(2k-a)\pi}{2M} \Big).
\end{equation}
Now combine (\ref{help1-cor-cot-1}) and (\ref{help2-cor-cot-1}) to obtain the desired formula.
\end{document}